\documentclass[11pt]{amsart}

\usepackage[T1]{fontenc}
\usepackage{lmodern}
\usepackage{amsmath,amssymb,amsthm,mathtools,mathrsfs}
\usepackage{booktabs}
\usepackage{enumitem}
\usepackage{microtype}
\usepackage{xcolor}
\usepackage[a4paper,margin=31mm]{geometry}
\usepackage[numbers,sort&compress]{natbib}
\usepackage[colorlinks=true,linkcolor=blue!55!black,citecolor=blue!55!black,urlcolor=blue!55!black]{hyperref}
\usepackage{doi}
\usepackage[nameinlink,noabbrev]{cleveref}
\hypersetup{
  pdftitle={Discrete q-Hermitian Clifford Analysis},
  pdfauthor={Baruch Schneider; Diana Schneiderova; Yifan Zhang},
  pdfsubject={Discrete q-Hermitian Clifford analysis on multiplicative q-lattices},
  pdfkeywords={discrete q-Hermitian Clifford analysis, Jackson derivative, multiplicative q-lattice, Fischer decomposition, Cauchy--Kovalevskaya extension, q-Fock space}
}

\allowdisplaybreaks
\newtheorem{theorem}{Theorem}[section]
\newtheorem{proposition}[theorem]{Proposition}

\newtheorem{corollary}[theorem]{Corollary}
\theoremstyle{definition}

\theoremstyle{remark}
\newtheorem{remark}[theorem]{Remark}

\newcommand{\C}{\mathbb C}
\newcommand{\N}{\mathbb N}
\newcommand{\Pcal}{\mathcal P}
\newcommand{\Fcal}{\mathcal F}
\newcommand{\Hcal}{\mathcal H}
\newcommand{\Mcal}{\mathcal M}
\newcommand{\Kcal}{\mathcal K}

\newcommand{\im}{\operatorname{im}}

\newcommand{\diag}{\operatorname{diag}}
\newcommand{\qnum}[1]{[#1]_q}
\newcommand{\qfac}[1]{[#1]_q!}
\newcommand{\Dz}{\mathcal D_q}
\newcommand{\Dw}{{\mathcal D_q^{\dagger}}}
\newcommand{\Z}{\mathcal Z}
\newcommand{\Zd}{{\mathcal Z^{\dagger}}}
\newcommand{\hZ}{\widehat{\mathcal Z}}
\newcommand{\hZd}{{\widehat{\mathcal Z}^{\dagger}}}
\newcommand{\hq}{\widehat\rho}
\newcommand{\eps}{\varepsilon}

\title[Discrete $q$-Hermitian Clifford analysis]
{Discrete $q$-Hermitian Clifford Analysis}

\author{Baruch Schneider${}^*$}{\thanks{$\ast$ Corresponding author}}
\address{Department of Mathematics, University of Ostrava, 70200 Ostrava, Czechia}
\email{baruch.schneider@osu.cz}

\author{Diana Schneiderov\'a}
\address{Department of Mathematics, University of Ostrava, 70200 Ostrava, Czechia}
\email{diana.schneiderova@osu.cz}

\author{Yifan Zhang}
\address{Department of Mathematics, University of Ostrava, 70200 Ostrava, Czechia}
\address{Department of Algebra, Charles University, 18675 Prague, Czechia}
\address{Department of Applied Mathematics, VSB--Technical University of Ostrava, 70800 Ostrava, Czechia}
\email{yifan.zhang@osu.cz}

\date{}

\begin{document}

\begin{abstract}
We develop a discrete $q$-Hermitian Clifford calculus based on coordinatewise Jackson differences on multiplicative $q$-lattices.  The resulting Hermitian Jackson--Dirac operators are nilpotent, factor a $q$-Laplacian, and have occupancy-dependent Euler anticommutators.  These give explicit Fischer projectors and homotopies, but the local calculus is not controlled by total degree.  We prove a two-face Cauchy--Kovalevskaya theorem for polynomials and extend it to a $q$-analytic Jackson--Fischer class.  A divided-power conjugation recovers scalar Euler relations and transfers the classical joint Fischer decomposition and dimension formulas.  The Jackson--Fischer completion is a vector-valued $q$-Fock space whose polarized nullspaces have projected reproducing kernels; the joint kernel is obtained by alternating projections.  We also determine the linear symmetry group of the multiplicative lattice.
\end{abstract}


\maketitle

\section{Introduction}\label{sec:introduction}

Hermitian Clifford analysis replaces the Euclidean Dirac operator by a pair of nilpotent first-order operators associated with the complex polarization.  Its basic operator calculus is developed in \cite{BrackxBuresDeSchepperEtAl2007I,BrackxBuresDeSchepperEtAl2007II,BrackxDeSchepperSommen2008}; a concise overview of the classical theory, including Fischer decomposition, Cauchy--Kovalevskaya extension, radial algebra, and integral formulas, is given by Sabadini and Sommen \cite{SabadiniSommen2015}.  We shall also use the classical Fischer decomposition and the Hermitian Cauchy--Kovalevskaya theorem from \cite{BrackxDeSchepperSoucek2010,BrackxDeSchepperLavickaSoucek2011}.  For the radial-algebra and Dirac-complex viewpoints, see \cite{Sommen1997,SabadiniStruppaSommenVanLancker2002,DeSchepperGuzmanSommen2017}; an integral formula of Martinelli--Bochner type was obtained in \cite{SommenPena2007}.

There are several inequivalent discretizations of Clifford analysis.  In the standard additive-lattice theory one uses forward and backward differences together with modified Clifford and Weyl relations; K\"ahler and Sommen \cite{KahlerSommen2015} give an overview including Fischer decomposition, Cauchy--Kovalevskaya extension, Taylor series, and a discrete Cauchy formula.  Related additive-lattice constructions include \cite{FaustinoKahlerSommen2007,DeRidderDeSchepperSommen2011,FaustinoKahler2015}.  A one-dimensional discrete Hermitian Weierstrass transform was studied in \cite{MasseSommenDeRidderRaeymaekers2021}.  Radial $q$-Dirac operators occur in \cite{CoulembierSommen2010,CoulembierSommen2011}, while quantum complex-space Laplacians use noncommuting coordinates and quantum-group covariance \cite{IorgovKlimyk2003}.  Most closely related to the present paper is the Jackson calculus of Zimmermann, Bernstein, and Schneider \cite{ZimmermannBernsteinSchneider2025}, where coordinatewise Jackson differences on commuting real variables lead to an orthogonal $q$-Dirac operator, Fischer decomposition, and a $q$-Cauchy--Kovalevskaya theorem.  The present construction is therefore a discrete $q$-calculus in the multiplicative-lattice sense: the shifts are coordinate dilations rather than translations on an additive lattice.  Accordingly, throughout this paper ``discrete'' refers to this multiplicative $q$-lattice setting, not to an additive $h$-lattice.

Here we organize coordinatewise Jackson differences into a Hermitian pair on
\[
 \C[z_1,\ldots,z_m,w_1,\ldots,w_m]\otimes\bigwedge\C^m,
\]
where $z$ and $w$ are algebraically independent.  With the usual creation and contraction operators $\eps_i$ and $\iota_i$, set
\[
 \Dz=\sum_{i=1}^m\eps_iD_{q,z_i},
 \qquad
 \Dw=\sum_{i=1}^m\iota_iD_{q,w_i}.
\]
Both operators square to zero and
\[
 \{\Dz,\Dw\}=\sum_{i=1}^mD_{q,z_i}D_{q,w_i}.
\]
The main difference from the classical Hermitian algebra appears when these operators are paired with the local coordinate variables
\[
 \Z=\sum_i\iota_i z_i,
 \qquad
 \Zd=\sum_i\eps_i w_i.
\]
Their same-polarization anticommutators are diagonal, but for $m\ge2$ they are not functions only of total polynomial and spinor degree.  On $z^\alpha w^\beta e_S$, the first one has eigenvalue
\[
 \sum_i[\alpha_i+\mathbf1_{i\in S}]_q,
\]
and the second has the analogous $w$-spectrum.  Dependence of the Jackson Euler eigenvalue on the distribution of polynomial degree already occurs in the orthogonal theory \cite{ZimmermannBernsteinSchneider2025}; the extra dependence on spin occupancy is specific to the Hermitian pair considered here.

These diagonal anticommutators give explicit one-polarization Fischer projectors and contracting homotopies.  They also determine the loss of full unitary symmetry: the linear normalizer of the coordinate-dilation group is the monomial group.  For the simultaneous equations $\Dz f=\Dw f=0$, we prove a two-face polynomial Cauchy--Kovalevskaya theorem, a branching formula for the boundary data, and a $q$-analytic extension theorem in a Jackson--Fischer coefficient space.  The two-face geometry itself is classical \cite{BrackxDeSchepperLavickaSoucek2011}; the new point is its realization for the Hermitian Jackson pair and the associated convergence estimates.

A divided-power conjugation gives a second, nonlocal coordinate pair for which the Euler relations again depend only on total degree.  Conjugating the classical Hermitian Fischer decomposition then gives a joint Jackson-monogenic decomposition and the classical multiplicity formulas.  The transported and local coordinate systems are different: the former has scalar Euler relations, whereas the latter consists of ordinary coordinate multiplications and Hilbert adjoints of the Jackson derivatives.  The same Jackson--Fischer inner product completes the polynomial space to a vector-valued $q$-Fock space and yields reproducing kernels for the polarized and joint nullspaces.

This construction is also different from the stable transported $q$-Hermitian coordinate calculus of \cite{SchneiderSchneiderovaZhang2026}.  That theory concerns transported Clifford--Weyl products, Capelli--PBW formulas, radial restriction, localization, and stable $U(m)$ covariance.  The present paper keeps the ordinary coordinate product and starts from explicit one-axis dilations on a multiplicative lattice.  Thus the two constructions address different operator algebras, even though both use divided powers at certain stages.

\section{Discrete Jackson calculus and the Hermitian Dirac pair}\label{sec:grid}

\subsection{The multiplicative \texorpdfstring{$q$}{q}-lattice and Jackson differences}

Throughout the paper,
\begin{equation}\label{eq:q-range}
 0<q<1,
 \qquad
 \qnum{r}:=\frac{1-q^r}{1-q},
 \qquad
 \qfac{r}:=\prod_{k=1}^r\qnum{k},
 \qquad \qfac{0}:=1.
\end{equation}
Let
\[
 \Pcal_m:=\C[z_1,\ldots,z_m,w_1,\ldots,w_m].
\]
The two sets of variables commute, and $w$ is kept independent of $z$ until a Hermitian slice is chosen.

For a coordinate $x$, let $T_{q,x}$ denote the dilation
\[
 (T_{q,x}f)(\ldots,x,\ldots)=f(\ldots,qx,\ldots).
\]
The one-sided Jackson derivative is
\begin{equation}\label{eq:jackson-derivative}
 D_{q,x}f:=\frac{T_{q,x}f-f}{(q-1)x}.
\end{equation}
At $x=0$ the right-hand side is interpreted by polynomial continuation.  Its elementary relations are
\begin{equation}\label{eq:jackson-relations}
 D_{q,x}x^r=\qnum{r}x^{r-1},
 \qquad
 D_{q,x}x=1+qxD_{q,x},
 \qquad
 D_{q,x}y=yD_{q,x}\quad(x\ne y).
\end{equation}
Jackson derivatives in different coordinates commute.

The difference calculus is governed by a twisted Leibniz rule.  Let
$f\mapsto f^\star$ be the conjugate-linear involution of $\Pcal_m$ determined
by $z_i^\star=w_i$ and $w_i^\star=z_i$.

\begin{proposition}[Jackson rules and involution]
\label{prop:jackson-product-rule}
For $x\in\{z_1,\ldots,z_m,w_1,\ldots,w_m\}$ and $f,g\in\Pcal_m$,
\begin{align}
 D_{q,x}(fg)
 &= (D_{q,x}f)g+(T_{q,x}f)(D_{q,x}g) \label{eq:left-q-Leibniz}\\
 &= (D_{q,x}f)(T_{q,x}g)+f(D_{q,x}g). \label{eq:right-q-Leibniz}
\end{align}
Moreover,
\begin{equation}\label{eq:dilation-commutation}
 D_{q,x}T_{q,x}=qT_{q,x}D_{q,x},
 \qquad [D_{q,x},T_{q,y}]=0\quad(x\ne y),
\end{equation}
and, for real $q$,
\begin{equation}\label{eq:star-jackson}
 (T_{q,z_i}f)^\star=T_{q,w_i}(f^\star),
 \qquad
 (D_{q,z_i}f)^\star=D_{q,w_i}(f^\star).
\end{equation}
Thus the two-polarization Jackson calculus is stable under the Hermitian
involution on the complexification.
\end{proposition}

\begin{proof}
Insert $T_{q,x}(fg)=(T_{q,x}f)(T_{q,x}g)$ in
\eqref{eq:jackson-derivative} and add and subtract either
$(T_{q,x}f)g$ or $f(T_{q,x}g)$.  This gives
\eqref{eq:left-q-Leibniz} and \eqref{eq:right-q-Leibniz}.  The relations in
\eqref{eq:dilation-commutation} follow directly from the definitions.  Since
$q$ is real and the involution exchanges $z_i$ with $w_i$, applying $\star$
to the difference quotient gives \eqref{eq:star-jackson}.
\end{proof}

The geometric meaning of \eqref{eq:jackson-derivative} is discrete.  If
$a=(a_1,\ldots,a_m)$ and $b=(b_1,\ldots,b_m)$ have nonzero coordinates, set
\begin{equation}\label{eq:multiplicative-lattice}
 \Lambda_q(a,b):=
 \left\{(q^{k_1}a_1,\ldots,q^{k_m}a_m,
 q^{\ell_1}b_1,\ldots,q^{\ell_m}b_m):
 k,\ell\in\mathbb Z^m\right\}.
\end{equation}
At every lattice point, $D_{q,z_i}$ compares the value there with the value at the adjacent point obtained by increasing $k_i$ by one; $D_{q,w_i}$ does the same in the conjugate polarization.  If $b=\overline a$, the diagonal sublattice defined by $\ell=k$ lies on the Hermitian slice $w=\overline z$ and is preserved by the paired dilations $T_{q,z_i}T_{q,w_i}$.  An individual dilation $T_{q,z_i}$ or $T_{q,w_i}$ does not preserve that slice; the separate Jackson operators are therefore operators on the complexification, exchanged by the involution $z_i^\star=w_i$.

The lattice \eqref{eq:multiplicative-lattice} is multiplicative rather than additive, and $q$ is both the ratio between adjacent nodes and the parameter in the Jackson relations.  This is different from the additive $h\mathbb Z^{2m}$ calculus reviewed in \cite{KahlerSommen2015}; see also \cite{FaustinoKahlerSommen2007,FaustinoKahler2015,MasseSommenDeRidderRaeymaekers2021}.

For multi-indices $\alpha,\beta\in\N_0^m$, write
\[
 z^\alpha w^\beta:=z_1^{\alpha_1}\cdots z_m^{\alpha_m}
 w_1^{\beta_1}\cdots w_m^{\beta_m},
 \quad
 \qfac{\alpha}:=\prod_i\qfac{\alpha_i},
 \quad
 \alpha!:=\prod_i\alpha_i!.
\]
We denote by $\Pcal_{r,s}$ the subspace of bidegree $(r,s)$.

\subsection{The Hermitian Jackson--Dirac pair}\label{sec:dirac-pair}

Let $\Fcal_m=\bigwedge\C^m$ with its standard orthonormal basis $e_S$, where
$S\subseteq\{1,\ldots,m\}$.  Exterior creation and contraction are denoted by
$\eps_i$ and $\iota_i=\eps_i^*$.  They satisfy
\begin{equation}\label{eq:car}
 \eps_i\eps_j+\eps_j\eps_i=0,
 \qquad
 \iota_i\iota_j+\iota_j\iota_i=0,
 \qquad
 \eps_i\iota_j+\iota_j\eps_i=\delta_{ij}I.
\end{equation}
Put
\begin{equation}\label{eq:occupancy}
 N_i:=\eps_i\iota_i,
 \qquad
 \mathsf N:=\sum_iN_i.
\end{equation}
Thus $N_i$ is the projection onto states containing $i$ and
$\mathsf Ne_S=|S|e_S$.

We work on the spinor-valued polynomial module
\[
 \mathscr P_m:=\Pcal_m\otimes\Fcal_m.
\]
The four odd operators are
\begin{equation}\label{eq:four-native-operators}
 \Dz:=\sum_{i=1}^m\eps_iD_{q,z_i},
 \qquad
 \Dw:=\sum_{i=1}^m\iota_iD_{q,w_i},
 \qquad
 \Z:=\sum_{i=1}^m\iota_i z_i,
 \qquad
 \Zd:=\sum_{i=1}^m\eps_i w_i.
\end{equation}
The first two are the Hermitian Jackson--Dirac pair; the last two are the local coordinate variables.  Their degree shifts on
\begin{equation}\label{eq:graded-sector}
 \mathscr P_{r,s}^n:=\Pcal_{r,s}\otimes\Fcal_m^n
\end{equation}
are
\begin{align*}
 \Dz&:(r,s,n)\mapsto(r-1,s,n+1),
 &\Z&:(r,s,n)\mapsto(r+1,s,n-1),\\
 \Dw&:(r,s,n)\mapsto(r,s-1,n-1),
 &\Zd&:(r,s,n)\mapsto(r,s+1,n+1).
\end{align*}
The dagger labels the conjugate polarization; Hilbert-space adjoints are denoted by $*$ and are computed in \cref{prop:fischer-adjoints}.  We write
\begin{equation}\label{eq:joint-monogenic-sector}
 \Mcal_{q;r,s}^n:=\ker\Dz\cap\ker\Dw\cap\mathscr P_{r,s}^n
\end{equation}
for the joint Jackson-monogenic sector.

Let $\mathfrak c$ denote conjugate polarization on first-order symbols.
It is conjugate-linear and exchanges
$z_i\leftrightarrow w_i$, $D_{q,z_i}\leftrightarrow D_{q,w_i}$, and
$\eps_i\leftrightarrow\iota_i$.  We use $\mathfrak c$ here only for this
algebraic exchange of the two polarizations; the Hilbert-space adjoint is
denoted by $*$.

\begin{proposition}[Conjugate polarization]\label{prop:hermitian-conjugacy}
The two Jackson--Dirac operators, and likewise the two local coordinate
variables, are exchanged by conjugate polarization:
\begin{equation}\label{eq:hermitian-conjugate-pairs}
 \mathfrak c(\Dz)=\Dw,\qquad \mathfrak c(\Z)=\Zd.
\end{equation}
\end{proposition}

\begin{proof}
Equation \eqref{eq:star-jackson} exchanges the two Jackson derivatives,
and conjugate polarization exchanges creation and contraction.  Applying
these rules term by term to \eqref{eq:four-native-operators} gives
\eqref{eq:hermitian-conjugate-pairs}.  This is the Jackson counterpart of
the conjugate pair of Hermitian Dirac operators reviewed in
\cite{SabadiniSommen2015}.
\end{proof}

\begin{theorem}[Local quadratic algebra]\label{thm:native-algebra}
The operators in \eqref{eq:four-native-operators} satisfy
\begin{align}
 \Dz^2&=\Dw^2=\Z^2=(\Zd)^2=0, \label{eq:nilpotence}\\
 \{\Z,\Zd\}&=\rho:=\sum_{i=1}^mz_iw_i, \label{eq:radius-factor}\\
 \{\Dz,\Dw\}&=\Delta_q:=\sum_{i=1}^mD_{q,z_i}D_{q,w_i}, \label{eq:laplace-factor}\\
 \{\Dz,\Zd\}&=0,
 &\{\Dw,\Z\}&=0. \label{eq:cross-zero}
\end{align}
The two remaining anticommutators are
\begin{align}
 A_q:=\{\Dz,\Z\}
 &=\sum_{i=1}^m\left(N_i+q^{N_i}E_{q,z_i}\right), \label{eq:Aq}\\
 B_q:=\{\Dw,\Zd\}
 &=\sum_{i=1}^m\left(1-N_i+q^{1-N_i}E_{q,w_i}\right), \label{eq:Bq}
\end{align}
where $E_{q,z_i}=z_iD_{q,z_i}$ and $E_{q,w_i}=w_iD_{q,w_i}$.
\end{theorem}

\begin{proof}
The square-zero identities follow from the commutativity of Jackson derivatives and coordinates together with \eqref{eq:car}.  In the two mixed products defining \eqref{eq:radius-factor} and \eqref{eq:laplace-factor}, the off-diagonal terms cancel and the diagonal terms reduce to the identity by the third relation in \eqref{eq:car}.  The cross relations \eqref{eq:cross-zero} contain two creation operators or two contractions, while the bosonic factors commute, so the same cancellation applies.

Only the diagonal terms survive in $\{\Dz,\Z\}$.  Using
$D_{q,z_i}z_i=1+qE_{q,z_i}$ and
$\iota_i\eps_i=1-N_i$, their sum is
\[
 \eps_i\iota_i(1+qE_{q,z_i})
 +\iota_i\eps_iE_{q,z_i}
 =N_i+\bigl(qN_i+1-N_i\bigr)E_{q,z_i}.
\]
Because $N_i$ is a projection,
$qN_i+1-N_i=q^{N_i}$, which gives \eqref{eq:Aq}.  The calculation for \eqref{eq:Bq} is identical with creation and contraction exchanged:
\[
 \iota_i\eps_i(1+qE_{q,w_i})
 +\eps_i\iota_iE_{q,w_i}
 =1-N_i+q^{1-N_i}E_{q,w_i}.
\]
\end{proof}

\begin{corollary}[Harmonicity and complexes]\label{cor:harmonicity}
The sequences defined by $\Dz$ and $\Dw$ are complexes, and every joint Jackson-monogenic polynomial is $q$-harmonic:
\[
 \ker\Dz\cap\ker\Dw\subseteq\ker\Delta_q.
\]
\end{corollary}

\begin{proof}
This is immediate from \eqref{eq:nilpotence} and \eqref{eq:laplace-factor}.
\end{proof}

\section{Local Fischer theory and lattice symmetry}\label{sec:homotopy}

\subsection{Euler operators and Fischer homotopies}

The operators $A_q$ and $B_q$ replace the two scalar Euler--spin operators of classical Hermitian Clifford analysis.  Their spectra are particularly simple.

\begin{proposition}[Euler spectra]\label{prop:euler-spectra}
For every monomial spinor $z^\alpha w^\beta e_S$,
\begin{align}
 A_q(z^\alpha w^\beta e_S)
 &=\left(\sum_{i=1}^m\qnum{\alpha_i+\mathbf1_{i\in S}}\right)
 z^\alpha w^\beta e_S, \label{eq:A-spectrum}\\
 B_q(z^\alpha w^\beta e_S)
 &=\left(\sum_{i=1}^m\qnum{\beta_i+1-\mathbf1_{i\in S}}\right)
 z^\alpha w^\beta e_S. \label{eq:B-spectrum}
\end{align}
Moreover,
\begin{equation}\label{eq:commute-homotopy}
 [A_q,\Dz]=[A_q,\Z]=0,
 \qquad
 [B_q,\Dw]=[B_q,\Zd]=0.
\end{equation}
The kernel of $A_q$ consists of the monomials with $\alpha=0$ and $S=\varnothing$; the kernel of $B_q$ consists of those with $\beta=0$ and $S=\{1,\ldots,m\}$.
\end{proposition}

\begin{proof}
Since $E_{q,z_i}z^\alpha=\qnum{\alpha_i}z^\alpha$ and
$N_ie_S=\mathbf1_{i\in S}e_S$, equation \eqref{eq:Aq} has eigenvalue
\[
 \mathbf1_{i\in S}+q^{\mathbf1_{i\in S}}\qnum{\alpha_i}
 =\qnum{\alpha_i+\mathbf1_{i\in S}}
\]
in the $i$th coordinate.  The proof of \eqref{eq:B-spectrum} uses
$1-\mathbf1_{i\in S}+q^{1-\mathbf1_{i\in S}}\qnum{\beta_i}
=\qnum{\beta_i+1-\mathbf1_{i\in S}}$.
The commutators in \eqref{eq:commute-homotopy} also follow abstractly from
$A_q=\Dz\Z+\Z\Dz$, $\Dz^2=\Z^2=0$, and the analogous identities for $B_q$.
All $q$-numbers in \eqref{eq:A-spectrum}--\eqref{eq:B-spectrum} are nonnegative and vanish only at zero, which gives the two kernels.
\end{proof}

Let $A_q^\#$ and $B_q^\#$ denote the diagonal generalized inverses obtained by inverting every positive eigenvalue in \eqref{eq:A-spectrum}--\eqref{eq:B-spectrum} and setting the exceptional zero eigenvalue to zero.

\begin{theorem}[Local orthogonal Fischer projectors]\label{thm:native-projectors}
Define
\begin{align}
 \Pi_z&:=I-\Z A_q^\#\Dz,
 &Q_z&:=\Z A_q^\#\Dz, \label{eq:z-projector}\\
 \Pi_w&:=I-\Zd B_q^\#\Dw,
 &Q_w&:=\Zd B_q^\#\Dw. \label{eq:w-projector}
\end{align}
Then $\Pi_z,Q_z$ and $\Pi_w,Q_w$ are pairs of complementary idempotents.  On every finite bidegree sector,
\begin{align}
 \im\Pi_z&=\ker\Dz,
 &\im Q_z&=\im\Z, \label{eq:z-ranges}\\
 \im\Pi_w&=\ker\Dw,
 &\im Q_w&=\im\Zd. \label{eq:w-ranges}
\end{align}
Consequently,
\begin{align}
 \mathscr P_{r,s}^n
 &=\bigl(\ker\Dz\cap\mathscr P_{r,s}^n\bigr)
 \oplus \Z\bigl(\ker\Dz\cap\mathscr P_{r-1,s}^{n+1}\bigr), \label{eq:z-fischer}\\
 \mathscr P_{r,s}^n
 &=\bigl(\ker\Dw\cap\mathscr P_{r,s}^n\bigr)
 \oplus \Zd\bigl(\ker\Dw\cap\mathscr P_{r,s-1}^{n-1}\bigr). \label{eq:w-fischer}
\end{align}
For the Jackson--Fischer inner product introduced in \cref{sec:fock}, these decompositions are orthogonal and the four operators in \eqref{eq:z-projector}--\eqref{eq:w-projector} are self-adjoint.
\end{theorem}

\begin{proof}
We prove the first pair.  By \eqref{eq:commute-homotopy}, the generalized inverse $A_q^\#$ commutes with $\Dz$ and $\Z$.  The exceptional kernel of $A_q$ is annihilated by both of these operators.  Hence, on the whole polynomial module,
\begin{align*}
 Q_z^2
 &=\Z(A_q^\#)^2\Dz\Z\Dz
 =\Z(A_q^\#)^2(A_q-\Z\Dz)\Dz
 =\Z A_q^\#\Dz=Q_z.
\end{align*}
It follows that $\Pi_z=I-Q_z$ is the complementary idempotent.  Moreover,
$\Dz\Pi_z=0$, and $\Pi_z$ fixes every element of $\ker\Dz$.  Thus
$\im\Pi_z=\ker\Dz$.  The range of $Q_z$ is contained in $\im\Z$, while
$Q_z\Z=\Z$ on the source of $\Z$; hence $\im Q_z=\im\Z$.
Finally, if $v$ lies in the source sector of $\Z$, then
$\Z v=\Z\Pi_zv$ because $\Z Q_z=0$.  This gives the refined second summand in \eqref{eq:z-fischer}.  The proof for $B_q,\Dw,\Zd$ is the same.  Orthogonality follows from the adjoint identities proved in \cref{prop:fischer-adjoints}.
\end{proof}

\begin{corollary}[Local little Fischer decomposition]\label{cor:little-fischer}
For every $r,s\ge0$ and $0\le n\le m$,
\begin{equation}\label{eq:little-fischer}
 \mathscr P_{r,s}^n
 =\Mcal_{q;r,s}^n\mathbin{\overset{\perp}{\oplus}}
 \left(
   \Z\mathscr P_{r-1,s}^{n+1}
   +\Zd\mathscr P_{r,s-1}^{n-1}
 \right),
\end{equation}
where a sector with a negative polynomial degree or an invalid spin degree
is understood to be zero.  The sum in parentheses need not be direct.
This is the local Jackson counterpart of the little Fischer decomposition
in classical Hermitian Clifford analysis; see
\cite{BrackxDeSchepperSoucek2010,SabadiniSommen2015}.
\end{corollary}

\begin{proof}
On the finite-dimensional sector $\mathscr P_{r,s}^n$,
\[
 (\ker\Dz\cap\ker\Dw)^\perp
 =(\ker\Dz)^\perp+(\ker\Dw)^\perp.
\]
By the Jackson--Fischer adjoint identities in
\cref{prop:fischer-adjoints}, the two summands are respectively the
ranges of $\Z$ and $\Zd$ from the source sectors displayed in
\eqref{eq:little-fischer}.  This proves the orthogonal decomposition.
\end{proof}

\begin{corollary}[Contracting homotopies]\label{cor:contracting-homotopies}
Outside the endpoint $\alpha=0,S=\varnothing$, the $\Dz$-complex is contracted by $\Z A_q^\#$.  Outside $\beta=0,S=\{1,\ldots,m\}$, the $\Dw$-complex is contracted by $\Zd B_q^\#$.
Consequently,
\begin{equation}\label{eq:polarized-cohomology}
 H(\mathscr P_m,\Dz)
 \simeq \C[w_1,\ldots,w_m]e_\varnothing,
 \qquad
 H(\mathscr P_m,\Dw)
 \simeq \C[z_1,\ldots,z_m]e_{\{1,\ldots,m\}}.
\end{equation}
\end{corollary}

\begin{proof}
The homotopy identities are
\[
 \Dz\Z A_q^\#+\Z A_q^\#\Dz=A_qA_q^\#,
 \qquad
 \Dw\Zd B_q^\#+\Zd B_q^\#\Dw=B_qB_q^\#.
\]
Each right-hand side is the identity off the corresponding exceptional
eigenspace in \cref{prop:euler-spectra}.  Those exceptional spaces are closed
under the relevant differential and cannot be boundaries because they lie at
the bottom, respectively top, spin endpoint.  This proves
\eqref{eq:polarized-cohomology}.
\end{proof}

For local multiplication and the Jackson--Fischer adjoint, the Euler spectrum is not controlled by total degree.

\begin{proposition}[Failure of scalar total-degree closure]\label{prop:non-scalar-euler}
Suppose $m\ge2$ and $q\ne1$.  The operator $A_q$ is not scalar on the fixed sector $\mathscr P_{2,0}^0$, and $B_q$ is not scalar on $\mathscr P_{0,2}^m$.  In particular, neither local mixed anticommutator can be expressed as a function only of the corresponding total polynomial degree and the total spinor number.
\end{proposition}

\begin{proof}
On $z_1^2e_\varnothing$ and $z_1z_2e_\varnothing$, formula \eqref{eq:A-spectrum} gives the respective eigenvalues
$\qnum2=1+q$ and $2\qnum1=2$.  They are unequal for $q\ne1$.  On the top-spin vectors $w_1^2e_{\{1,\ldots,m\}}$ and
$w_1w_2e_{\{1,\ldots,m\}}$, equation \eqref{eq:B-spectrum} gives the same two eigenvalues.  Both pairs have equal total polynomial and spinor degrees.
\end{proof}

\begin{remark}[One complex dimension]\label{rem:rank-one}
For $m=1$, write $f=f_0+\eps_1f_1$.  Then
\[
 \Dz f=\eps_1D_{q,z}f_0,
 \qquad
 \Dw f=D_{q,w}f_1.
\]
Hence $f$ is joint monogenic exactly when $f_0$ is independent of $z$ and
$f_1$ is independent of $w$.  On a homogeneous spinor sector the two Euler
spectra are $A_q=\qnum{r+n}$ and $B_q=\qnum{s+1-n}$.  The dependence on the
distribution among coordinate axes therefore begins at $m=2$.  The companion
coordinate introduced in \cref{sec:transport} acts by
$\widehat z z^r=(r+1)z^{r+1}/\qnum{r+1}$.
\end{remark}

\begin{remark}[First nonscalar sector]\label{rem:first-anisotropic}
For $m=2$ and bottom spin, $z_1^2e_\varnothing$ and
$z_1z_2e_\varnothing$ have the same total degree but $A_q$-eigenvalues
$1+q$ and $2$.  The corresponding homotopy coefficients are
$(1+q)^{-1}$ and $1/2$.
\end{remark}

\subsection{Linear symmetry of the multiplicative lattice}\label{sec:symmetry}

Let $T_i\in GL_m(\C)$ be the linear dilation which multiplies the $i$th coordinate by $q$ and fixes the remaining coordinates, and put
\[
 \Gamma_q:=\langle T_1,\ldots,T_m\rangle
 =\{\diag(q^{k_1},\ldots,q^{k_m}):k\in\mathbb Z^m\}.
\]
We call $\Gamma_q$ the coordinate-dilation group of the $z$-lattice.

\begin{theorem}[Normalizer of the coordinate-dilation group]\label{thm:grid-normalizer}
Let $0<q<1$ and $L\in GL_m(\C)$.  The following are equivalent:
\begin{enumerate}[label=(\roman*),leftmargin=2.1em]
 \item $L\Gamma_qL^{-1}=\Gamma_q$;
 \item conjugation by $L$ permutes the set $\{T_1,\ldots,T_m\}$;
 \item $L$ is a monomial matrix, that is, a product of an invertible diagonal matrix and a permutation matrix.
\end{enumerate}
Consequently, the maximal unitary subgroup normalizing the coordinate-dilation group is
\[
 U(1)^m\rtimes\mathfrak S_m,
\]
not $U(m)$ when $m\ge2$.  The paired $(z,w)$ grid has the corresponding Hermitian action
$z\mapsto Lz$, $w\mapsto\overline Lw$ for monomial unitary $L$.
\end{theorem}

\begin{proof}
Suppose first that $L$ normalizes $\Gamma_q$.  The conjugate $LT_iL^{-1}$ lies in $\Gamma_q$ and has the same spectrum as $T_i$, namely one eigenvalue $q$ and $m-1$ eigenvalues $1$.  The only elements of $\Gamma_q$ with this spectrum are the $T_j$, so conjugation by $L$ permutes their set.  The $q$-eigenspace of $T_i$ is the coordinate line $\C e_i$, while its $1$-eigenspace is the complementary coordinate hyperplane.  If
$LT_iL^{-1}=T_{\sigma(i)}$, then $L$ maps $\C e_i$ onto
$\C e_{\sigma(i)}$.  Hence every column of $L$ has exactly one nonzero entry and the induced map $i\mapsto\sigma(i)$ is a permutation.  Thus $L$ is monomial.  Conversely, a diagonal matrix commutes with every $T_i$ and a permutation matrix permutes them, so every monomial matrix normalizes $\Gamma_q$.  Intersecting the monomial group with $U(m)$ gives the displayed subgroup.
\end{proof}

\begin{remark}[Covariance versus normalization]\label{rem:grid-covariance}
The theorem concerns the underlying lattice and is independent of spinors.  Permutations and diagonal unitary changes can be lifted to the exterior module, so the Hermitian Jackson pair is covariant under the displayed monomial unitary group.  A generic unitary mixing of two coordinates does not preserve the family of one-axis dilations.  The non-scalar spectrum in \cref{prop:non-scalar-euler} is the corresponding operator-level obstruction.
\end{remark}

\section{A two-face \texorpdfstring{$q$}{q}-Cauchy--Kovalevskaya theorem}\label{sec:ck}

The joint system requires data on both polarizations.  Its two-face form is the Jackson analogue of the classical Hermitian CK scheme in \cite{BrackxDeSchepperLavickaSoucek2011}.  We first determine the compatible polynomial data and the resulting finite recursion.  Separate the last complex coordinate and write
\begin{equation}\label{eq:split-dirac}
 \Dz=\Dz'+\eps_mD_{q,z_m},
 \qquad
 \Dw=\Dw'+\iota_mD_{q,w_m},
\end{equation}
where the primed operators use the first $m-1$ coordinates.  Every spinor polynomial has a unique form
\begin{equation}\label{eq:spin-split}
 f=A+\eps_mB,
 \qquad
 A,B\in\C[z,w]\otimes\bigwedge\C^{m-1}.
\end{equation}
Because $m$ is the last exterior label, the canonical anticommutation relations give
\begin{align}
 \Dz f=0
 &\Longleftrightarrow
 \Dz'A=0,
 \quad D_{q,z_m}A=\Dz'B, \label{eq:CK-system-one}\\
 \Dw f=0
 &\Longleftrightarrow
 \Dw'B=0,
 \quad D_{q,w_m}B=-\Dw'A. \label{eq:CK-system-two}
\end{align}

Use Jackson divided powers
\begin{equation}\label{eq:jackson-divided-powers}
 z_m^{\langle r\rangle_q}:=\frac{z_m^r}{\qfac r},
 \qquad
 w_m^{\langle s\rangle_q}:=\frac{w_m^s}{\qfac s},
\end{equation}
so that the corresponding Jackson derivative lowers the index with coefficient one.

\begin{theorem}[Joint Hermitian Jackson--CK extension]\label{thm:joint-CK}
Let
\begin{align}
 a(w_m)&=\sum_{s\ge0}w_m^{\langle s\rangle_q}a_s,
 &\Dz'a_s&=0, \label{eq:a-data}\\
 b(z_m)&=\sum_{r\ge0}z_m^{\langle r\rangle_q}b_r,
 &\Dw'b_r&=0, \label{eq:b-data}
\end{align}
be polynomial data with coefficients in
$\C[z_1,\ldots,z_{m-1},w_1,\ldots,w_{m-1}]
\otimes\bigwedge\C^{m-1}$.
There is a unique joint Jackson-monogenic polynomial
$f=A+\eps_mB$ satisfying
\begin{equation}\label{eq:two-face-boundary}
 A\big|_{z_m=0}=a(w_m),
 \qquad
 B\big|_{w_m=0}=b(z_m).
\end{equation}
Writing
\begin{equation}\label{eq:AB-expansion}
 A=\sum_{r,s\ge0}z_m^{\langle r\rangle_q}w_m^{\langle s\rangle_q}A_{r,s},
 \qquad
 B=\sum_{r,s\ge0}z_m^{\langle r\rangle_q}w_m^{\langle s\rangle_q}B_{r,s},
\end{equation}
the extension is determined by
\begin{align}
 A_{0,s}&=a_s,
 &B_{r,0}&=b_r, \label{eq:CK-boundary-coefficients}\\
 A_{r+1,s}&=\Dz'B_{r,s},
 &B_{r,s+1}&=-\Dw'A_{r,s}. \label{eq:CK-recursion}
\end{align}
Only finitely many coefficients in \eqref{eq:AB-expansion} are nonzero.
\end{theorem}

\begin{proof}
Expanding \eqref{eq:CK-system-one}--\eqref{eq:CK-system-two} in the divided-power basis gives precisely \eqref{eq:CK-recursion}.  The two face restrictions prescribe the coefficients in \eqref{eq:CK-boundary-coefficients}, so the recursion determines every remaining coefficient uniquely.

It remains to verify the tangential equations.  The boundary assumptions give
$\Dz'A_{0,s}=0$ and $\Dw'B_{r,0}=0$.  The first recurrence and
$(\Dz')^2=0$ imply $\Dz'A_{r,s}=0$ for $r>0$; the second recurrence and
$(\Dw')^2=0$ imply $\Dw'B_{r,s}=0$ for $s>0$.  Thus all four equations in \eqref{eq:CK-system-one}--\eqref{eq:CK-system-two} hold.

For finiteness, each diagonal step away from either boundary applies one of the operators
$-\Dz'\Dw'$ or $-\Dw'\Dz'$.  These lower the total tangential polynomial degree by two.  Since the boundary series are polynomials, every chain terminates and only finitely many coefficients survive.
\end{proof}

For the next statement, write $\mathscr P_{m-1;r,s}^n$ for the tangential
sector in the first $m-1$ coordinate pairs, and let $d_{m;r,s}^n$ denote the
dimension of $\Mcal_{q;r,s}^n$ in rank $m$.

\begin{corollary}[Boundary-data isomorphism and branching]\label{cor:CK-branching}
On each bidegree and spin sector, the boundary map in
\eqref{eq:two-face-boundary} is an isomorphism
\begin{align}
 \Mcal_{q;r,s}^n \longrightarrow{}
 &\bigoplus_{j=0}^{s}
 \bigl(\ker\Dz'\cap\mathscr P_{m-1;r,s-j}^n\bigr) \notag\\
 &\quad\oplus
 \bigoplus_{i=0}^{r}
 \bigl(\ker\Dw'\cap\mathscr P_{m-1;r-i,s}^{n-1}\bigr),
 \label{eq:CK-boundary-isomorphism}
\end{align}
where sectors with invalid spin degree are zero.  In particular,
\begin{align}
 d_{m;r,s}^n={}&
 \sum_{j=0}^{s}\dim\bigl(\ker\Dz'\cap\mathscr P_{m-1;r,s-j}^n\bigr)
 \notag\\*
 &+\sum_{i=0}^{r}\dim\bigl(\ker\Dw'\cap
 \mathscr P_{m-1;r-i,s}^{n-1}\bigr).
 \label{eq:CK-dimension-branching}
\end{align}
Thus the two faces contain neither redundant nor missing polynomial data.
\end{corollary}

\begin{proof}
For a polynomial of bidegree $(r,s)$, the coefficient $a_j$ in
\eqref{eq:a-data} belongs to $\mathscr P_{m-1;r,s-j}^n$, and the coefficient
$b_i$ in \eqref{eq:b-data} belongs to
$\mathscr P_{m-1;r-i,s}^{n-1}$.  The compatibility conditions place them in
the displayed kernels.  The existence and uniqueness in
\cref{thm:joint-CK} show that every such tuple has exactly one preimage.
Taking dimensions gives \eqref{eq:CK-dimension-branching}.
\end{proof}

\begin{corollary}[Closed diagonal recursions]\label{cor:closed-CK}
For $r,s\ge1$,
\begin{equation}\label{eq:closed-diagonal-CK}
 A_{r,s}=-\Dz'\Dw'A_{r-1,s-1},
 \qquad
 B_{r,s}=-\Dw'\Dz'B_{r-1,s-1}.
\end{equation}
More explicitly,
\begin{align}
 A_{r,s}
 &=
 \begin{cases}
  (-\Dz'\Dw')^r a_{s-r},&r\le s,\\
  (-\Dz'\Dw')^s\Dz'b_{r-s-1},&r>s,
 \end{cases} \label{eq:closed-A}\\
 B_{r,s}
 &=
 \begin{cases}
  (-\Dw'\Dz')^s b_{r-s},&s\le r,\\
  (-\Dw'\Dz')^r(-\Dw'a_{s-r-1}),&s>r.
 \end{cases} \label{eq:closed-B}
\end{align}
Thus the extension can be evaluated by finite alternating tangential Dirac
iterations from the nearer boundary face.
\end{corollary}

\begin{proof}
Substitution of \eqref{eq:CK-recursion} into itself gives
\eqref{eq:closed-diagonal-CK}.  If $r\le s$, repeated diagonal descent
from $A_{r,s}$ reaches $A_{0,s-r}=a_{s-r}$; if $r>s$, it reaches
$A_{r-s,0}=\Dz'b_{r-s-1}$.  The two alternatives for $B_{r,s}$ follow
in the same way, using $B_{0,k}=-\Dw'a_{k-1}$ for $k>0$.
\end{proof}

\subsection{A rank-two example}\label{subsec:rank-two-CK}

The recursion is already nontrivial for $m=2$.  Take total spin degree one and prescribe
\[
 a(w_2)=w_2z_1w_1\eps_1,
 \qquad
 b(z_2)=0.
\]
The tangential conditions hold because
$\Dz'(z_1w_1\eps_1)=0$ and $\Dw'0=0$.  The only nonzero coefficients are
\[
 A_{0,1}=z_1w_1\eps_1,
 \qquad
 B_{0,2}=-z_1,
 \qquad
 A_{1,2}=-\eps_1.
\]
Thus
\begin{equation}\label{eq:rank-two-CK-example}
 f=
 z_1w_1w_2\eps_1
 -\frac{z_2w_2^2}{\qnum2}\eps_1
 -\frac{z_1w_2^2}{\qnum2}\eps_2.
\end{equation}
Since $D_{q,w_2}(w_2^2/\qnum2)=w_2$, direct substitution gives
$\Dz f=\Dw f=0$.  The factor $\qnum2^{-1}=(1+q)^{-1}$ is fixed by the Jackson recursion; at $q=1$ this becomes the classical divided-power extension.

\subsection{Extension to a \texorpdfstring{$q$}{q}-analytic Jackson--Fischer class}\label{subsec:analytic-ck}

The same recursion also defines infinite divided-power series.  We prove
convergence in the coefficient Hilbert space below.  This is the centered
version of the power-series notion of $q$-analyticity used in the one-operator
Jackson theory \cite{ZimmermannBernsteinSchneider2025}; at the centre $0$ the
$q$-shifted powers used there reduce to ordinary powers.

Let $\mathscr X_q^{(m-1)}$ be the completion of tangential spinor polynomials
in the first $m-1$ coordinate pairs for the norm
\begin{equation}\label{eq:tangential-fock-norm}
 \left\|\sum_{\alpha,\beta,S}c_{\alpha,\beta,S}
 z'^{\alpha}w'^{\beta}e_S\right\|_{\mathscr X_q}^2
 :=\sum_{\alpha,\beta,S}|c_{\alpha,\beta,S}|^2
 \qfac{\alpha}\qfac{\beta}.
\end{equation}
Each tangential Jackson derivative is a bounded weighted shift on this
space.  Consequently $\Dz'$ and $\Dw'$ are bounded, and we put
\begin{equation}\label{eq:CK-mu}
 \mu_{q,m}:=
 \max\{\|\Dz'\Dw'\|,\|\Dw'\Dz'\|\}.
\end{equation}
The elementary triangle estimate gives
\begin{equation}\label{eq:CK-mu-bound}
 \mu_{q,m}\le \frac{(m-1)^2}{1-q}.
\end{equation}
For a Banach space $X$ and $R>0$, write
\begin{equation}\label{eq:AqRX}
 \mathscr A_q(R;X):=
 \left\{u(x)=\sum_{j\ge0}x^{\langle j\rangle_q}u_j:
 \|u\|_{q,R}:=\sum_{j\ge0}\frac{R^j}{\qfac j}\|u_j\|_X<\infty\right\}.
\end{equation}
Finally, set
\begin{equation}\label{eq:Phi-q}
 \Phi_q(t):=\sum_{k=0}^{\infty}\frac{t^k}{(\qfac k)^2},
 \qquad 0\le t<(1-q)^{-2}.
\end{equation}

\begin{theorem}[$q$-analytic two-face CK extension]\label{thm:analytic-CK}
Let $m\ge2$, let $R_z,R_w>0$, and suppose
\begin{align*}
 a(w_m)&=\sum_{s\ge0}w_m^{\langle s\rangle_q}a_s
 \in\mathscr A_q(R_w;\mathscr X_q^{(m-1)}),
 &\Dz'a_s&=0,\\
 b(z_m)&=\sum_{r\ge0}z_m^{\langle r\rangle_q}b_r
 \in\mathscr A_q(R_z;\mathscr X_q^{(m-1)}),
 &\Dw'b_r&=0.
\end{align*}
Define $A_{r,s},B_{r,s}$ by
\eqref{eq:CK-boundary-coefficients}--\eqref{eq:CK-recursion}.  If
$0<\rho_z<R_z$ and $0<\rho_w<R_w$ satisfy
\begin{equation}\label{eq:analytic-CK-domain}
 \mu_{q,m}\rho_z\rho_w<(1-q)^{-2},
\end{equation}
then the two series in \eqref{eq:AB-expansion} converge normally in
$\mathscr X_q^{(m-1)}$ for
$|z_m|\le\rho_z$, $|w_m|\le\rho_w$.  Their sum
$f=A+\eps_mB$ is joint Jackson-monogenic and has the prescribed two-face
data.  It is the unique joint solution in this normally convergent divided-power
class.

Moreover, for $|z_m|\le\rho_z$ and $|w_m|\le\rho_w$,
\begin{align}
 \|A(z_m,w_m)\|_{\mathscr X_q}
 &\le \Phi_q(\mu_{q,m}\rho_z\rho_w)
 \left(\|a\|_{q,\rho_w}
       +\rho_z\|\Dz'\|\,\|b\|_{q,\rho_z}\right),
 \label{eq:analytic-CK-A-estimate}\\
 \|B(z_m,w_m)\|_{\mathscr X_q}
 &\le \Phi_q(\mu_{q,m}\rho_z\rho_w)
 \left(\|b\|_{q,\rho_z}
       +\rho_w\|\Dw'\|\,\|a\|_{q,\rho_w}\right).
 \label{eq:analytic-CK-B-estimate}
\end{align}
Thus the extension depends continuously on both boundary data on every
smaller bidisc satisfying \eqref{eq:analytic-CK-domain}.
\end{theorem}

\begin{proof}
The one-coordinate Jackson derivative is a weighted backward shift with
weights $\sqrt{\qnum r}$ for the norm
\eqref{eq:tangential-fock-norm}; hence its norm is at most
$(1-q)^{-1/2}$.  Summing the $m-1$ tangential terms proves
\eqref{eq:CK-mu-bound}.

We use two elementary estimates.  First,
\begin{equation}\label{eq:qfactorial-supermultiplicative}
 \qfac{r+t}\ge \qfac r\,\qfac t,
 \qquad r,t\ge0,
\end{equation}
because $\qnum{r+j}\ge\qnum j$ for $j\ge1$.  Second,
$\lim_{k\to\infty}(\qfac k)^{1/k}=(1-q)^{-1}$, so the series
\eqref{eq:Phi-q} has radius $(1-q)^{-2}$.

Use now the closed formulas \eqref{eq:closed-A}--\eqref{eq:closed-B}.
For the part of $A$ coming from the $a$-face, write $s=r+t$.  Then
\eqref{eq:qfactorial-supermultiplicative} gives
\begin{align*}
 &\sum_{r,t\ge0}
 \frac{\rho_z^r\rho_w^{r+t}}
      {\qfac r\,\qfac{r+t}}
 \|(\Dz'\Dw')^r a_t\|_{\mathscr X_q}\\
 &\qquad\le
 \Phi_q(\mu_{q,m}\rho_z\rho_w)
 \sum_{t\ge0}\frac{\rho_w^t}{\qfac t}
 \|a_t\|_{\mathscr X_q}.
\end{align*}
For the part coming from the $b$-face, put $r=s+t+1$ and use
$\qfac{s+t+1}\ge\qfac{s+1}\qfac t$ to obtain
\begin{align*}
 &\sum_{s,t\ge0}
 \frac{\rho_z^{s+t+1}\rho_w^s}
      {\qfac{s+t+1}\,\qfac s}
 \|(\Dz'\Dw')^s\Dz'b_t\|_{\mathscr X_q}\\
 &\qquad\le
 \rho_z\|\Dz'\|\,
 \Phi_q(\mu_{q,m}\rho_z\rho_w)
 \sum_{t\ge0}\frac{\rho_z^t}{\qfac t}
 \|b_t\|_{\mathscr X_q}.
\end{align*}
This proves normal convergence of $A$ and
\eqref{eq:analytic-CK-A-estimate}.  The same argument with the roles of
$z$ and $w$ interchanged proves normal convergence of $B$ and
\eqref{eq:analytic-CK-B-estimate}.

Because the inequalities in the hypotheses are strict, choose
$\rho_z<\sigma_z<R_z$ and $\rho_w<\sigma_w<R_w$ such that
$\mu_{q,m}\sigma_z\sigma_w<(1-q)^{-2}$.  The preceding estimates give
normal convergence on the larger bidisc with radii $\sigma_z,\sigma_w$.
The normal-direction Jackson derivatives may therefore be applied term by term on the
smaller bidisc with radii $\rho_z,\rho_w$, while the tangential operators
are bounded.  The recurrence then yields
\eqref{eq:CK-system-one}--\eqref{eq:CK-system-two}.  The boundary
values are immediate.  Finally, any normally convergent divided-power
solution has uniquely determined coefficients; comparing them in the two
first-order equations reproduces \eqref{eq:CK-recursion}.  This proves
uniqueness and the stated continuous dependence.
\end{proof}

\begin{remark}[Size of the analytic neighborhood]\label{rem:analytic-CK-domain}
The condition \eqref{eq:analytic-CK-domain} is a sufficient convergence
condition obtained from operator norms, not a claim of maximality.  In
particular, using \eqref{eq:CK-mu-bound} it is enough that
\[
 \rho_z\rho_w<\frac{1}{(m-1)^2(1-q)}.
\]
For polynomial boundary data the recursion terminates, so no restriction of
this kind is needed and \cref{thm:joint-CK} is recovered.
\end{remark}

\begin{remark}[Two-face boundary data]\label{rem:two-faces}
A single value on the codimension-two intersection $z_m=w_m=0$ does not determine the joint first-order system.  The component $A$ is propagated in the $z_m$ direction from $z_m=0$, whereas $B$ is propagated in the $w_m$ direction from $w_m=0$.  Conditions \eqref{eq:a-data}--\eqref{eq:b-data} are exactly the tangential compatibility conditions.  Their geometry is inherited from the classical Hermitian system \cite{BrackxDeSchepperLavickaSoucek2011}; the divided powers turn its differential recursion into the present Jackson recursion.  The one-operator Jackson--CK problem in \cite{ZimmermannBernsteinSchneider2025} and the additive-lattice extension in \cite{DeRidderDeSchepperSommen2011} use different boundary schemes.
\end{remark}

\section{Divided-power companions and the joint Fischer decomposition}\label{sec:transport}

The coordinate variables $\Z$ and $\Zd$ are lattice-local and are the Hilbert adjoints of the Jackson derivatives.  A second coordinate pair is useful for the joint algebra.  Define the diagonal divided-power transform
\begin{equation}\label{eq:Gq}
 G_q(z^\alpha w^\beta e_S)
 :=\frac{\alpha!\,\beta!}{\qfac\alpha\,\qfac\beta}
 z^\alpha w^\beta e_S.
\end{equation}
It is an automorphism of each finite polynomial sector.  If
$\partial_{z_i},\partial_{w_i}$ are the ordinary derivatives, then
\begin{equation}\label{eq:derivative-intertwining}
 D_{q,z_i}=G_q\partial_{z_i}G_q^{-1},
 \qquad
 D_{q,w_i}=G_q\partial_{w_i}G_q^{-1}.
\end{equation}

\begin{corollary}[Intertwining of CK extension]\label{cor:CK-intertwining}
Let $\operatorname{CK}_q$ denote the polynomial boundary-to-solution map
of \cref{thm:joint-CK}.  Let $\operatorname{CK}_1$ denote its $q=1$
specialization in the same creation--contraction convention.  After the two
face polynomials are expanded coefficientwise, this is the classical
Hermitian polynomial extension of
\cite{BrackxDeSchepperLavickaSoucek2011}.  Let $G_q^\partial$ be the
restriction of $G_q$ to the two boundary faces.  Then
\begin{equation}\label{eq:CK-intertwining}
 G_q\operatorname{CK}_1
 =\operatorname{CK}_qG_q^\partial.
\end{equation}
Hence the Jackson theorem retains the two-face initial-data geometry of the
classical Hermitian system.
\end{corollary}

\begin{proof}
The transform $G_q$ conjugates the two classical Dirac operators to the
Jackson pair by \eqref{eq:derivative-intertwining}, commutes with the spinor
split \eqref{eq:spin-split}, and restricts to $G_q^\partial$ on the two
faces.  Both sides of \eqref{eq:CK-intertwining} are therefore joint
Jackson-monogenic and have the same boundary data.  Uniqueness in
\cref{thm:joint-CK} proves the identity.
\end{proof}

\begin{remark}[Form of the CK operator]\label{rem:CK-formula-size}
For a fixed bidegree, the classical Hermitian CK extension in
\cite{BrackxDeSchepperLavickaSoucek2011} is written as an explicit finite
operator sum over the two families of admissible initial polynomials.  In
\cref{thm:joint-CK,cor:closed-CK} the same two families are collected into
the face polynomials $a(w_m)$ and $b(z_m)$.  Thus the shorter display does
not reduce the amount of Cauchy data.  The Jackson divided powers remove
the factorial coefficients from the recursion, leading to
\eqref{eq:closed-A}--\eqref{eq:closed-B}.  After expansion of the face data,
the specialization $q=1$ is the classical construction in the present
creation--contraction convention, while \eqref{eq:CK-intertwining} gives
its Jackson counterpart.  We shall use this recursive form.
\end{remark}

Let
\begin{equation}\label{eq:companion-coordinates}
 \widehat z_i:=G_qz_iG_q^{-1},
 \qquad
 \widehat w_i:=G_qw_iG_q^{-1}.
\end{equation}
Their monomial action is
\begin{align}
 \widehat z_i(z^\alpha w^\beta e_S)
 &=\frac{\alpha_i+1}{\qnum{\alpha_i+1}}
 z^{\alpha+e_i}w^\beta e_S, \label{eq:zhat-action}\\
 \widehat w_i(z^\alpha w^\beta e_S)
 &=\frac{\beta_i+1}{\qnum{\beta_i+1}}
 z^\alpha w^{\beta+e_i}e_S. \label{eq:what-action}
\end{align}
Set
\begin{equation}\label{eq:companion-Hermitian-variables}
 \hZ:=\sum_i\iota_i\widehat z_i,
 \qquad
 \hZd:=\sum_i\eps_i\widehat w_i,
 \qquad
 \hq:=\sum_i\widehat z_i\widehat w_i.
\end{equation}

\begin{theorem}[Transported Hermitian algebra]\label{thm:transported-algebra}
Let $E_z=\sum_i z_i\partial_{z_i}$ and
$E_w=\sum_i w_i\partial_{w_i}$.  Then
\begin{align}
 \hZ^2&=(\hZd)^2=0,
 &\{\hZ,\hZd\}&=\hq, \label{eq:transport-coordinate}\\
 \{\Dz,\hZd\}&=0,
 &\{\Dw,\hZ\}&=0, \label{eq:transport-cross}\\
 \{\Dz,\hZ\}&=E_z+\mathsf N,
 &\{\Dw,\hZd\}&=E_w+m-\mathsf N. \label{eq:transport-euler}
\end{align}
Moreover,
\begin{equation}\label{eq:transport-laplacian}
 \Delta_q=G_q\Delta G_q^{-1},
 \qquad
 \Delta:=\sum_i\partial_{z_i}\partial_{w_i}.
\end{equation}
\end{theorem}

\begin{proof}
All formulas are the conjugates by $G_q$ of the standard Hermitian Clifford relations.  The Euler operators commute with $G_q$ because all three are diagonal on monomials.  Equation \eqref{eq:transport-laplacian} also follows directly from \eqref{eq:derivative-intertwining}.
\end{proof}

The local and companion coordinate systems have different roles.

\begin{proposition}[Local and companion coordinates]\label{prop:local-spectral-tradeoff}
For $q\ne1$, $\widehat z_i$ differs from multiplication by $z_i$ on every monomial with positive $z_i$-degree, and the analogous statement holds for $\widehat w_i$.  The local pair $(\Z,\Zd)$ consists of coordinate multiplications and satisfies
$\Dz^*=\Z$, $\Dw^*=\Zd$ for the Jackson--Fischer inner product.  The companion pair $(\hZ,\hZd)$ has the scalar Euler relations \eqref{eq:transport-euler}, but its coefficients depend on the individual coordinate degrees through \eqref{eq:zhat-action}--\eqref{eq:what-action}.  In particular, it is diagonal in the degree variables before raising the corresponding coordinate, rather than ordinary multiplication on \eqref{eq:multiplicative-lattice}.
\end{proposition}

\begin{proof}
For $r\ge1$, the factor $(r+1)/\qnum{r+1}$ equals one only when $q=1$.  The remaining assertions follow from \cref{prop:fischer-adjoints,thm:transported-algebra}.
\end{proof}

Put
\begin{equation}\label{eq:harmonic-space}
 \Hcal_{q;r,s}^n:=\ker\Delta_q\cap\mathscr P_{r,s}^n.
\end{equation}
For $a,b\ge0$ and $0\le n\le m$, define
\begin{equation}\label{eq:double-lift}
 \widehat X_{a,b,n}:=
 (a+n)\hZ\hZd-(b+m-n)\hZd\hZ.
\end{equation}

\begin{theorem}[Joint Jackson--Hermitian Fischer decomposition]\label{thm:joint-fischer}
For every $r,s\ge0$ and $0\le n\le m$,
\begin{align}
 \Hcal_{q;r,s}^n
 ={}&\Mcal_{q;r,s}^n
 \oplus\hZ\Mcal_{q;r-1,s}^{n+1}
 \oplus\hZd\Mcal_{q;r,s-1}^{n-1} \notag\\
 &\oplus\mathbf1_{1\le n\le m-1}
 \widehat X_{r-1,s-1,n}\Mcal_{q;r-1,s-1}^n, \label{eq:harmonic-joint-fischer}
\end{align}
where a summand with a negative polynomial degree or an invalid spin degree is zero.  Iterating the scalar harmonic decomposition gives
\begin{equation}\label{eq:full-joint-fischer}
 \mathscr P_{r,s}^n
 =\bigoplus_{j=0}^{\min(r,s)}
 \hq^j\,\mathscr R_{q;r-j,s-j}^n,
\end{equation}
where $\mathscr R_{q;r,s}^n$ denotes the four-term right-hand side of
\eqref{eq:harmonic-joint-fischer}.
\end{theorem}

\begin{proof}
At $q=1$, formula \eqref{eq:harmonic-joint-fischer} is the classical Hermitian harmonic Fischer decomposition; see \cite{BrackxDeSchepperSoucek2010}.  Its double-coordinate coefficient can also be checked directly.  If
$u\in\Mcal_{1;a,b}^n$, put $\alpha=a+n$ and $\beta=b+m-n$.  The classical mixed algebra gives
\[
 \Delta(\Z\Zd u)=\beta u,
 \qquad
 \Delta(\Zd\Z u)=\alpha u,
\]
so $\alpha\Z\Zd u-\beta\Zd\Z u$ is harmonic.  The standard directness proof uses the scalar Euler homotopies.

Conjugate the complete classical direct sum by $G_q$.  Equations
\eqref{eq:derivative-intertwining} and \eqref{eq:companion-Hermitian-variables} carry the classical Dirac pair, coordinate pair, radius, harmonic kernel, and joint kernel to their hatted Jackson counterparts.  This proves \eqref{eq:harmonic-joint-fischer}.  Conjugating and iterating the classical scalar Fischer decomposition
$\Pcal_{r,s}=\ker\Delta\oplus\rho\Pcal_{r-1,s-1}$ gives \eqref{eq:full-joint-fischer}.
\end{proof}

\begin{corollary}[Joint monogenic multiplicities]\label{cor:joint-dimensions}
Let $d_{m;r,s}^n=\dim\Mcal_{q;r,s}^n$.  These dimensions are independent of
$q\in(0,1)$.  For $1\le n\le m-1$,
\begin{equation}\label{eq:dimension-formula}
 d_{m;r,s}^n=
 \frac{(r+m-1)!(s+m-1)!(r+s+m)}
 {r!s!(m-1)!(n-1)!(m-n-1)!(r+n)(s+m-n)}.
\end{equation}
At the spin endpoints,
\begin{equation}\label{eq:endpoint-dimensions}
 d_{m;r,s}^0=\delta_{r0}\binom{s+m-1}{m-1},
 \qquad
 d_{m;r,s}^m=\delta_{s0}\binom{r+m-1}{m-1}.
\end{equation}
\end{corollary}

\begin{proof}
The transform $G_q$ is an isomorphism from each classical joint kernel onto the corresponding Jackson joint kernel, so the dimensions agree with the classical ones.  Formula \eqref{eq:dimension-formula} is the Weyl dimension of the classical Hermitian monogenic module (see also \cite{BrackxDeSchepperSoucek2010}) with highest weight
$(r+1,1^{n-1},0^{m-n-1},-s)$.  At $n=0$, the equation $\Dz f=0$ forces independence of $z$, while $\Dw$ vanishes on bottom spin; the top-spin statement is dual.
\end{proof}

\begin{remark}[Local and transported structures]\label{rem:transport-scope}
The joint nullspace and its decomposition concern the explicit Jackson--Dirac pair.  The direct-sum coordinates in \eqref{eq:harmonic-joint-fischer}, however, are the degree-spectral companions.  The orthogonal one-polarization splittings \eqref{eq:z-fischer}--\eqref{eq:w-fischer} use the local coordinate variables.  For $m\ge2$, \cref{prop:non-scalar-euler} shows that the scalar Euler relations of the companion pair do not hold for the local coordinate multiplications.
\end{remark}

\section{The Jackson--Fischer space and reproducing kernels}\label{sec:fock}

Define an inner product on $\mathscr P_m$ by declaring the monomial spinors orthogonal and setting
\begin{equation}\label{eq:fischer-inner-product}
 \left\langle z^\alpha w^\beta e_S,
 z^{\alpha'}w^{\beta'}e_T\right\rangle_q
 :=\delta_{\alpha\alpha'}\delta_{\beta\beta'}\delta_{ST}
 \qfac\alpha\,\qfac\beta.
\end{equation}
With $m$ replaced by $m-1$, this is exactly the tangential coefficient norm
\eqref{eq:tangential-fock-norm} used in the analytic CK theorem.  Thus the
extension in \cref{thm:analytic-CK} and the reproducing-kernel completion below
belong to the same Hilbert scale.

\begin{proposition}[Jackson--Fischer adjoints]\label{prop:fischer-adjoints}
On polynomials,
\begin{equation}\label{eq:fischer-adjoints}
 D_{q,z_i}^*=z_i,
 \qquad
 D_{q,w_i}^*=w_i,
 \qquad
 \Dz^*=\Z,
 \qquad
 \Dw^*=\Zd.
\end{equation}
Consequently, $A_q$ and $B_q$ are positive self-adjoint operators, and the projectors in \cref{thm:native-projectors} are orthogonal.
\end{proposition}

\begin{proof}
For one coordinate,
\[
 \left\langle D_{q,x}x^r,x^{r-1}\right\rangle_q
 =\qnum r\qfac{r-1}=\qfac r
 =\left\langle x^r,xx^{r-1}\right\rangle_q.
\]
Tensoring this identity over the coordinates and using $\eps_i^*=\iota_i$ gives \eqref{eq:fischer-adjoints}.  Positivity follows from
$A_q=\Dz\Dz^*+\Dz^*\Dz$ and the analogous identity for $B_q$.
\end{proof}

\begin{corollary}[Local scalar harmonic Fischer decomposition]\label{cor:native-harmonic-fischer}
On every bidegree and spin sector,
\begin{equation}\label{eq:native-harmonic-fischer}
 \mathscr P_{r,s}^n
 =\Hcal_{q;r,s}^n\mathbin{\overset{\perp}{\oplus}}
 \rho\,\mathscr P_{r-1,s-1}^n.
\end{equation}
If $r,s>0$, the orthogonal harmonic projector is
\begin{equation}\label{eq:native-harmonic-projector}
 \Pi_{\mathrm{har}}
 =I-\rho\,(\Delta_q\rho)^{-1}\Delta_q,
\end{equation}
where $\Delta_q\rho$ is restricted to $\mathscr P_{r-1,s-1}^n$.
Consequently, the full joint decomposition also has the local radial form
\begin{equation}\label{eq:local-radial-joint-fischer}
 \mathscr P_{r,s}^n
 =\bigoplus_{j=0}^{\min(r,s)}
 \rho^j\mathscr R_{q;r-j,s-j}^n.
\end{equation}
\end{corollary}

\begin{proof}
Equation \eqref{eq:fischer-adjoints} and the two factorizations in
\cref{thm:native-algebra} give
\[
 \Delta_q^*=\{\Dz,\Dw\}^*
 =\{\Z,\Zd\}=\rho.
\]
Multiplication by the nonzero polynomial $\rho$ is injective.  In finite
bidegree, $\Delta_q$ is therefore surjective onto its target, and
$\im\rho=(\ker\Delta_q)^\perp$.  This proves
\eqref{eq:native-harmonic-fischer}.  The operator $\Delta_q\rho$ is positive
definite on the source sector because
$\langle\Delta_q\rho u,u\rangle_q=\|\rho u\|_q^2$; the standard formula for
the orthogonal projection along $\im\rho$ gives
\eqref{eq:native-harmonic-projector}.  Iterating
\eqref{eq:native-harmonic-fischer} and applying
\cref{thm:joint-fischer} to each harmonic term proves
\eqref{eq:local-radial-joint-fischer}.
\end{proof}

\begin{remark}[Two radial complements]\label{rem:two-radial-complements}
Equations \eqref{eq:full-joint-fischer} and
\eqref{eq:local-radial-joint-fischer} are different direct decompositions of
the same polynomial sector.  The former is the divided-power transport of the
classical decomposition and uses $\hq=G_q\rho G_q^{-1}$.  The latter uses
ordinary multiplication by the lattice-local radius $\rho$ and is orthogonal
at the scalar harmonic step.  The two radial variables agree at $q=1$ but not
in positive degree when $q\ne1$.
\end{remark}

Let $\mathscr H_q^{(m)}$ be the Hilbert completion.  Put
\begin{equation}\label{eq:q-exponential}
 \mathrm e_q(t):=\sum_{r=0}^\infty\frac{t^r}{\qfac r},
 \qquad
 R_q:=(1-q)^{-1/2},
\end{equation}
and let
\begin{equation}\label{eq:q-polydisc}
 \Omega_q:=\{(z,w)\in\C^{2m}:|z_i|<R_q, |w_i|<R_q\text{ for all }i\}.
\end{equation}
The one-variable space with weights $\qfac r$ is the generalized $q$-Fock space studied in \cite{AlpayCerejeirasKahlerSchneider2024}.  The tensor-product kernel of our spinor space is
\begin{equation}\label{eq:full-kernel}
 K_q((z,w),(\zeta,\omega))
 =\prod_{i=1}^m\mathrm e_q(z_i\overline{\zeta_i})
  \mathrm e_q(w_i\overline{\omega_i})\,I_{\Fcal_m}.
\end{equation}

\begin{theorem}[$q$-Fock realization]\label{thm:q-fock-realization}
The completion $\mathscr H_q^{(m)}$ is a reproducing-kernel Hilbert space of
$\Fcal_m$-valued holomorphic functions on $\Omega_q$ with kernel
\eqref{eq:full-kernel}.  All four local operators in
\eqref{eq:four-native-operators} extend boundedly to this space.  The generalized inverses $A_q^\#,B_q^\#$ and the four Fischer projectors also extend boundedly, and
\begin{align}
 \Kcal_z&:=\ker\Dz,
 &\Kcal_w&:=\ker\Dw,
 &\Mcal_q&:=\ker\Dz\cap\ker\Dw
\end{align}
are closed subspaces.
\end{theorem}

\begin{proof}
The orthonormal basis consists of
$z^\alpha w^\beta e_S/(\qfac\alpha\qfac\beta)^{1/2}$.  Summing its rank-one kernels gives \eqref{eq:full-kernel}.  Since
$\qnum r\to(1-q)^{-1}$, the diagonal series converges exactly when each
$|z_i|,|w_i|<R_q$.

On the normalized one-variable basis, both $D_{q,x}$ and multiplication by
$x$ are weighted shifts with weights $\sqrt{\qnum r}$ and
$\sqrt{\qnum{r+1}}$, respectively.  Their norms are at most
$(1-q)^{-1/2}$.  The exterior operators have norm one, so the finite sums in \eqref{eq:four-native-operators} are bounded.  The positive nonzero eigenvalues of $A_q$ and $B_q$ are at least one by \eqref{eq:A-spectrum}--\eqref{eq:B-spectrum}; hence
$\|A_q^\#\|,\|B_q^\#\|\le1$.  The projector formulas are therefore bounded, and kernels of bounded operators are closed.
\end{proof}

\begin{theorem}[Polarized and joint reproducing kernels]\label{thm:projected-kernels}
The polarized spaces $\Kcal_z$ and $\Kcal_w$ have reproducing kernels
\begin{align}
 K_{q,z}(x,y)&=\Pi_z^{(x)}K_q(x,y), \label{eq:z-kernel}\\
 K_{q,w}(x,y)&=\Pi_w^{(x)}K_q(x,y), \label{eq:w-kernel}
\end{align}
where $x=(z,w)$ and the indicated projector acts in the first variable.
Let
\begin{equation}\label{eq:alternating-projection}
 P_{q,\mathrm{mon}}
 :=\operatorname*{s-lim}_{k\to\infty}(\Pi_z\Pi_w)^k.
\end{equation}
Then $P_{q,\mathrm{mon}}$ is the orthogonal projection onto $\Mcal_q$, and
\begin{equation}\label{eq:joint-kernel}
 K_{q,\mathrm{mon}}(x,y)
 :=P_{q,\mathrm{mon}}^{(x)}K_q(x,y)
\end{equation}
is the reproducing kernel of the joint Jackson-monogenic $q$-Fock space.
\end{theorem}

\begin{proof}
If $P$ is an orthogonal projection in a reproducing-kernel Hilbert space, the range of $P$ has kernel $P^{(x)}K(x,y)$.  Apply this to the orthogonal projectors from \cref{thm:native-projectors,prop:fischer-adjoints} to obtain
\eqref{eq:z-kernel}--\eqref{eq:w-kernel}.  Von Neumann's alternating projection theorem (see, for example, \cite{Bauschke2001}) gives the strong limit in \eqref{eq:alternating-projection} and identifies it with the orthogonal projection onto
$\im\Pi_z\cap\im\Pi_w=\ker\Dz\cap\ker\Dw$.  Applying the same range-kernel principle proves \eqref{eq:joint-kernel}.
\end{proof}

\begin{remark}[Joint kernel]\label{rem:kernel-algorithm}
Equations \eqref{eq:z-kernel}--\eqref{eq:w-kernel} are explicit spectral formulas because the eigenvalues of $A_q$ and $B_q$ are given in
\eqref{eq:A-spectrum}--\eqref{eq:B-spectrum}.  Formula \eqref{eq:joint-kernel} gives the joint kernel by a convergent sequence of projections.  We do not obtain a closed basic-hypergeometric expression for it.
\end{remark}

\paragraph{Classical limit.}
For every fixed polynomial degree,
\begin{equation}\label{eq:classical-limit}
 D_{q,x}\longrightarrow\partial_x,
 \quad
 A_q\longrightarrow E_z+\mathsf N,
 \quad
 B_q\longrightarrow E_w+m-\mathsf N,
 \quad
 G_q\longrightarrow I
 \qquad(q\to1).
\end{equation}
Thus the local and companion coordinates agree in the limit, the Euler
spectrum depends only on total degree and spin degree, the two-face recursion
becomes the classical Hermitian CK recursion, and the reproducing kernel tends
coefficientwise to
\[
 \exp(z\cdot\overline\zeta+w\cdot\overline\omega)I_{\Fcal_m}.
\]

\section{Conclusion}\label{sec:conclusion}

Coordinatewise Jackson differences give a discrete $q$-Hermitian Dirac pair on a multiplicative $q$-lattice.  The pair is nilpotent and factors the mixed $q$-Laplacian.  Its anticommutators with the local coordinate multiplications have positive occupancy-dependent spectra, which give orthogonal Fischer projectors and contracting homotopies.  In rank at least two these spectra are not functions of total degree, in agreement with the reduction of the linear symmetry group to the monomial unitary group.

For the simultaneous null system we obtained polynomial and $q$-analytic two-face Cauchy--Kovalevskaya extensions.  A divided-power conjugation gives companion coordinates with the classical scalar Euler relations and hence transports the joint Hermitian Fischer decomposition and multiplicity formulas.  With the Jackson--Fischer inner product, the polynomial space completes to a vector-valued $q$-Fock space; the polarized kernels are obtained from the Fischer projectors and the joint kernel from alternating projections.  Closed formulas for the joint kernel and boundary integral formulas remain open.

\medskip
\noindent\textbf{Funding.}
This work was co-funded by the University of Ostrava, Grant No.~SGS05/P\v{R}F/2026. Yifan Zhang is also co-funded by the European Union under the REFRESH -- Research Excellence For REgion Sustainability and High-tech Industries project, No.~CZ.10.03.01/00/22003/0000048, via the Operational Programme Just Transition.

\medskip
\section*{Statements and Declarations}

\noindent\textbf{Competing interests.}
The authors declare that they have no competing financial or non-financial interests that are directly or indirectly related to this work.
\medskip

\noindent\textbf{Data availability.}
No datasets were generated or analysed during the current study.

\begingroup
\hbadness=10000
\bibliographystyle{plainnat}
\bibliography{references}
\endgroup

\end{document}